\documentclass{mcom-l}

\newtheorem{theorem}{Theorem}[section]
\newtheorem{lemma}[theorem]{Lemma}
\newtheorem{cor}[theorem]{Corollary}

\numberwithin{equation}{section}

\begin{document}

\title{Improvements to Turing's Method}

\author{Timothy Trudgian}
\address{Mathematical Institute, University of Oxford, OX1 3LB England}
\email{trudgian@maths.ox.ac.uk}
\thanks{I wish to acknowledge the financial support of the General Sir John Monash Foundation, and Merton College, Oxford.}


\date{December 7, 2009.}


\begin{abstract}
This article improves the estimate of the size of the definite integral of $S(t)$, the argument of the Riemann zeta-function. The primary application of this improvement is Turing's Method for the Riemann zeta-function. Analogous improvements are given for the arguments of Dirichlet $L$-functions and of Dedekind zeta-functions.
\end{abstract}

\maketitle

\section{Introduction}\label{TuringIntro}

In determining the number of non-trivial zeroes of the Riemann zeta-function $\zeta(s)$ in a given range, one proceeds in two stages. First, one can compute a number of zeroes along the critical line using Gram's Law\footnote{Briefly: The Gram points $\{g_{n}\}_{n\geq -1}$ are easily computed and have an average spacing equal to that of the non-trivial zeroes of $\zeta(\sigma +it)$, viz. $g_{n+1}-g_{n} \asymp (\log g_{n})^{-1}$. Gram's Law states that for $t\in(g_{n}, g_{n+1}]$ there is exactly one zero of $\zeta(\frac{1}{2}+it)$. Gram's Law was shown to fail infinitely often by Titchmarsh \cite{Titchmarsh3},  and shown to fail in a positive proportion of cases by the author \cite{Trudgian}.}  or Rosser's Rule\footnote{For some $n$, one defines a Gram block of length $p$ as the interval $(g_{n}, g_{n+p}]$, wherein there is an even number of zeroes in each of the intervals $(g_{n}, g_{n+1}]$ and $(g_{n+p-1}, g_{n+p}]$, and an odd number of zeroes in each of the intervals $(g_{n+1}, g_{n+2}], \ldots, (g_{n+p-2}, g_{n+p-1}]$. Rosser's Rule then states that a Gram block of length $p$ contains exactly $p$ zeroes of $\zeta(\frac{1}{2}+it)$. Rosser's Rule holds more frequently than Gram's Law, but its infinite failure was first shown by Lehman \cite{Lehman} --- see also the work of the author [\textit{op.\ cit.}]} (see, e.g. \cite[Chs VI-VII]{Edwards}), which gives one a lower bound on the total number of zeroes in the critical strip in that range. To conclude that one has found the \textit{precise} number of zeroes in this range, one needs an additional argument.

The earliest method employed was due to Backlund \cite{Backlund} and relies on showing that $\Re\zeta(s) \neq 0$ along the lines connecting $2, 2+iT, \frac{1}{2}+iT$. This is very labour intensive: nevertheless Backlund was able to perform these procedure for $T=200$ and later Hutchinson extended this to $T=300.468$ (see \cite[\textit{loc.\ cit.}]{Edwards} for more details). In both cases the zeroes of $\zeta(s)$ located via Gram's Law were verified to be the only zeroes in the given ranges. Titchmarsh \cite{Titchmarsh3} continued to use this method to show that the Riemann hypothesis is valid for $|t|\leq 1468$.

Asides from its computational intricacies, this method of Backlund is bound to fail for sufficiently large $T$. To see this, it is convenient to introduce the function $S(T)$ defined as
\begin{equation}\label{sdef}
S(T) = \pi^{-1} \arg \zeta(\tfrac{1}{2}+iT),
\end{equation}
where if $T$ is not an ordinate of a zero of $\zeta(s)$, the argument is determined by continuous variation along the lines connecting $2, 2+iT, \frac{1}{2}+ iT.$ If $T$ coincides with a zero of $\zeta(s)$ then
\begin{equation*}
S(T) = \lim_{\delta\rightarrow 0} \{S(t+\delta) + S(t-\delta)\}.
\end{equation*}
The interest in the properties of $S(T)$ is immediate once one considers its relation to the function $N(T)$, the number of non-trivial zeroes of $\zeta(\sigma+it)$ for $|t|\leq T$. In the equation
\begin{equation}
N(T) = \frac{T}{2\pi} \log\frac{T}{2\pi} - \frac{T}{2\pi} + \frac{7}{8} + O(T^{-1}) + S(T),
\end{equation}
the error term is continuous in $T$, whence it follows that $S(T)$ increases by +1 whenever $T$ passes over a zero of the zeta-function. Concerning the behaviour of $S(T)$ are the following estimates
\begin{equation}\label{littlewood}
\int_{0}^{T} S(t)\, dt = O(\log T),
\end{equation}
due to Littlewood (see, e.g. \cite[pp.\ 221-222]{Titchmarsh}), and
\begin{equation}\label{Sgrowth}
S(T) = \Omega_{\pm}\left(\frac{(\log T)^{\frac{1}{3}}}{(\log\log T)^{\frac{7}{3}}}\right),
\end{equation}
due to Selberg \cite{Selberg1}.

Returning to Backlund's approach: if $\Re\zeta(s) \neq 0$ along the lines connecting $2, 2+iT, \frac{1}{2}+iT$ then, when varied along these same lines, $|\arg\zeta(s)| < \frac{\pi}{2},$ if one takes the principal argument. It therefore follows that $|S(T)|< \frac{1}{2}$, whence $S(T)$ is bounded, which contradicts (\ref{Sgrowth}).
\subsection{Turing's Method}

A more efficient procedure in producing an upper bound on the number of zeroes in a given range was proposed by Turing \cite{Turing} in 1953. This relies on a quantitative version of Littlewood's result (\ref{littlewood}), given below as
\begin{theorem}\label{TurCri}
Given $t_{0}>0$, there are positive constants $a$ and $b$ such that, for $t_{2}>t_{1}>t_{0}$, the following estimate holds
\begin{equation}\label{Tur82}
\bigg|\int_{t_{1}}^{t_{2}} S(t)\, dt \bigg| \leq a + b \log t_{2}.
\end{equation}
\end{theorem}

Since $S(t)$ increases by $+1$ whenever $t$ passes over a zero (on the line or not), the existence of too many zeroes in the range $t\in(t_{1}, t_{2})$ would cause the integral in (\ref{Tur82}) to be too large. 

Turing's paper \cite{Turing} contains several errors, which are fortunately corrected by Lehman \cite{Lehman}. Furthermore, Lehman also improves the constants $a$ and $b$, thereby making Turing's Method more easily applicable. Here additional improvements on the constants in Turing's Method are given in \S \ref{TM}. Rumely \cite{Rumely} has adapted Turing's Method to Dirichlet $L$-functions and this is herewith improved in \S \ref{TMDLF}. Finally, in \S \ref{TMDZF} the analogous improvements to the argument of Dedekind zeta-functions is given, following the work of Tollis \cite{Tollis}.

It is interesting to note the motivation of Turing as he writes in \cite[p.\ 99]{Turing}
\begin{quote}
The calculations were done in an optimistic hope that a zero would be found off the critical line, and the calculations were directed more towards finding such zeros than proving that none existed. 
\end{quote}
Indeed, Turing's Method has become the standard technique used in modern verification of the Riemann hypothesis.

\section{Turing's Method for the Riemann zeta-function}\label{TM}
\subsection{New results}\label{New Results}
In general let the triple of numbers ($a, b, t_{0}$) satisfy Theorem \ref{TurCri}. Turing showed that $(2.07, 0.128, 168\pi)$ satisfied (\ref{Tur82}) and Lehman showed that $(1.7, 0.114, 168\pi)$ does so as well. Brent \cite[Thm 2]{Brent} used the result of Lehman [\textit{op.\ cit.}, Thm 4] to prove the following
\begin{theorem}[Lehman--Brent]\label{LBrent}
If $N$ consecutive Gram blocks with union $[g_{n}, g_{p})$ satisfy Rosser's Rule, where
\begin{equation}\label{GB}
N\geq \frac{b}{6\pi}\log^{2} g_{p}+ \frac{(a-b\log 2\pi)}{6\pi}\log g_{p},
\end{equation}
then
\begin{equation*}
N(g_{n}) \leq n+1;\quad N(g_{p}) \geq p+1.
\end{equation*}
\end{theorem}
Since, by assumption these $N$ Gram blocks together contain exactly $p-n$ zeroes, this shows that up to height $g_{p}$ there are at most $p+1$ zeroes: and this is precisely the upper bound one has sought. Using the constants of Lehman viz.\ $(a=1.7, b=0.114)$ it is seen that one must find at least
\begin{equation}\label{BrentN}
N \geq 0.0061 \log^{2} g_{p} + 0.08 \log g_{p}
\end{equation}
consecutive Gram blocks to apply Theorem \ref{LBrent}. This constraint on $N$ has been used in the modern computational search for zeroes, and appears in the early works, e.g.\ \cite{Brent} right through to the recent, e.g.\ \cite{XG}. 

Turing makes the remark several times in his paper \cite{Turing} that the constant $b$ could be improved at the expense of the constant $a$. In (\ref{GB}) the first term dominates when $g_{p}$ is large, and therefore for computation at a large height it is desirable to choose $b$ to be small. Indeed, what is sought is the minimisation of
\begin{equation}\label{Ftimes}
F(a, b, g_{p}) = b\log\frac{g_{p}}{2\pi} +a.
\end{equation}
Current verification of the Riemann hypothesis has surpassed the height $T = 10^{12}$, see, e.g.\ \cite{XG} wherein (\ref{BrentN}) requires the location of at least 8 Gram blocks. In \S \ref{TuringCalc} the function $F(a, b, g_{p})$ is minimised at $g_{p} = 2\pi \cdot 10^{12}$ which leads to
\begin{theorem}\label{Thm1}
If $t_{2} >t_{1} >168\pi$, then
\begin{equation*}
\bigg|\int_{t_{1}}^{t_{2}} S(t)\, dt\bigg| \leq 2.067 + 0.059 \log t_{2}.
\end{equation*}
\end{theorem}
It should be noted that the constants achieved in Theorem \ref{Thm1} are valid\footnote{The constant $168\pi$ which occurs in the triples of Turing and Lehman seems to be a misprint. In the proof of the rate of growth of $\zeta(\frac{1}{2} +it)$, given here in Lemma \ref{l2}, Turing and Lehman require $t>128\pi$ so that the error terms in the Riemann--Siegel formula are small. A computational check shows that Lemma \ref{l2} in fact holds for all $t>1$. Choosing a moderately large value of $t_{0}$ ensures that the small errors accrued (i.e.\ the $\delta$ in Lemma \ref{l44} and the $\epsilon$ in Lemma \ref{l9}) are suitably small. At no point do Turing and Lehman require the imposition of a $t_{0}$ greater than $128\pi$. It is worthwhile to note that one could replace $168\pi$ in Theorem \ref{Thm1} by a smaller number, and although this has little application to the zeta-function, it may be useful for future applications to Dedekind zeta-functions --- cf.\ \S \ref{TMDZF}.} for all $t_{2}>t_{1}>168\pi$, and that at $t_{1}>2\pi \cdot 10^{12}$ these constants minimise the right-side of (\ref{Ftimes}). The above theorem and Theorem \ref{LBrent} immediately lead to
\begin{cor}\label{Thm2}
If $N$ consecutive Gram blocks with union $[g_{n}, g_{p})$ satisfy Rosser's Rule where
\begin{equation*}
N \geq 0.0031 \log^{2} g_{p} + 0.11 \log g_{p},
\end{equation*}
then
\begin{equation*}
N(g_{n}) \leq n+1; \qquad N(g_{p}) \geq p+1.
\end{equation*}
\end{cor}
The above corollary shows that, in order to apply Turing's Method at height $g_{p} = 2\pi\cdot 10^{12}$, one needs to find only 6 Gram blocks in which the Rosser Rule is valid.
\subsection{Proof of Theorem \ref{Thm1}}\label{Proof of Theorem 1.2.2}
This section closely follows the structure of Lehman's refinement \cite{Lehman} of Turing's work \cite{Turing}. Some of the Lemmas are identical to those in these papers, and their proofs are deferred to \cite{Lehman}.
To begin, one rewrites the integral of the function $S(t)$ using the following 
\begin{lemma}\label{Ll1}
If $t_{2} >t_{1} >0$, then
\begin{equation}\label{110}
\pi\int_{t_{1}}^{t_{2}} S(t)\, dt =\Re\int_{\frac{1}{2}+it_{2}}^{\infty +it_{2}} \log \zeta(s) \, ds - \Re\int_{\frac{1}{2}+it_{1}}^{\infty +it_{1}} \log \zeta(s) \, ds.
\end{equation}
\end{lemma}
\begin{proof}
This is Lemma 1 in \cite{Lehman}, and the proof is based on Littlewood's theorem for analytic functions, but more detail is supplied in \cite{LehmanOld}, or \cite[pp.\ 190-192]{Edwards}.
\end{proof}
Henceforth: Lemmas \ref{l2}--\ref{200} are used to bound the first integral on the right-hand side of (\ref{110}), and Lemmas \ref{l8}--\ref{l9} are needed to bound the second integral.

\begin{lemma}\label{l2}
If $t\geq 128\pi$, then
\begin{equation*}
|\zeta\left(\tfrac{1}{2} +it\right)| \leq 2.53\,  t^{\frac{1}{4}}.
\end{equation*}
\end{lemma}
\begin{proof}
See the argument in \cite{Lehman} where some corrections are given to Titchmarsh's explicit calculation of the error in the Riemann--Siegel formula.
\end{proof}
This estimate can certainly be improved insofar as reducing the exponent of $t$ is concerned. Currently the best bound on the growth of the zeta-function is due to Huxley \cite{Huxley}, viz.\ $\zeta(\frac{1}{2} +it) \ll t^{\alpha +\epsilon},$ where $\alpha = \frac{32}{205} \approx 0.1561$. However the methods used to attain this bound are complicated and the calculation of the implied constant would prove lengthy. The coarser, but simpler proof (see \cite[{Ch.\ V} \S 5]{Titchmarsh}) due to van der Corput yields
\begin{equation}\label{vdc}
|\zeta\left(\tfrac{1}{2} +it\right)| \leq At^{\frac{1}{6}} \log t,
\end{equation}
where the calculation of the constant $A$ is not too time consuming. Indeed following the arguments in \cite[{Chs IV-V}]{Titchmarsh} and using a result of Karatsuba \cite[Lem.\ 1]{KKapprox}, one can take $A\leq 20$. 

The logarithmic term in (\ref{vdc}) is relatively innocuous since, for a given $\eta>0$ one can then take $t_{0}$ so large that
\begin{equation*}
\log t \leq A' t^{\eta},
\end{equation*}
where $A' = A'(\eta, t_{0})$ can be easily computed, whence
\begin{equation*}
|\zeta\left(\tfrac{1}{2} +it\right)| \leq AA't^{\frac{1}{6} +\eta}.
\end{equation*}
Turing \cite[p.\ 108]{Turing} makes reference to the improvements made possible by these refined estimates on the growth of $\zeta(\frac{1}{2} +it)$. The following Lemmas will be written with 
\begin{equation}\label{currentl}
|\zeta\left(\tfrac{1}{2} +it\right)| \leq K t^{\theta},
\end{equation}
so that the benefit of such a refinement as that in (\ref{vdc}) can be seen clearly.

The bound on $\zeta(s)$ on the line $\sigma=\frac{1}{2}$ can be combined with that on the line $\sigma = c >1$, whence the Phragm\'{e}n--Lindel\"{o}f theorem can be applied throughout the strip $\frac{1}{2}\leq \sigma \leq c$. The papers of Turing and Lehman use the value $c=\frac{5}{4}$ and some improvement will be given later by choosing an optimal value of $c$ at the end of the proof. A result needed is
\begin{lemma}\label{l5}
Let $a, b, Q$ and $k$ be real numbers, and let $ f(s)$ be regular analytic in the strip $-Q\leq a\leq \sigma \leq b$ and satisfy the growth condition 
$$ |f(s)| <C\exp \left\{e^{k|t|}\right\},$$ for a certain $C>0$ and for $0<k<\pi/(b-a)$. Also assume that
\[|f(s)|\leq\left\{\begin{array}{ll}
A|Q+s|^{\alpha} &  \mbox{for $\Re(s) = a,$}\\
B|Q+s|^{\beta} & \mbox{for $\Re(s) = b$}\\
\end{array}
\right.\]
with $\alpha \geq \beta$. Then throughout the strip $a \leq \sigma \leq b$ the following holds
$$ |f(s)| \leq A^{(b-\sigma)/(b-a)}B^{(\sigma - a)/(b-a)}|Q+s|^{\alpha(b-\sigma)/(b-a) + \beta(\sigma-a)/(b-a)}.$$
\end{lemma}
\begin{proof}
See \cite[pp.\ 66-67]{Rad}.
\end{proof}
Take $Q=0;\; a=\frac{1}{2};\; b=c;\; f(s) = (s-1)\zeta(s)$, whence all the conditions of Lemma \ref{l5} are satisfied. Then on the line $\sigma = \frac{1}{2}$ it follows that
\begin{equation*}
|f(s)| \leq Kt^{\theta} |s-1| \leq K|s|^{\theta+1},
\end{equation*}
by virtue of (\ref{currentl}). On the line $\sigma = c$,
\begin{equation*}
|f(s)| \leq |s-1|\zeta(c) \leq \zeta(c) |s|,
\end{equation*}
since $c>1$. So one can take $A= K; \; \alpha = \theta +1;\; B=\zeta(c);\; \beta = 1$ and then apply Lemma \ref{l5} to obtain,
\begin{equation}\label{1.16}
|(s-1)\zeta(s)| \leq\left[ K^{c-\sigma}\{\zeta(c)\}^{\sigma-\frac{1}{2}} |s|^{\theta(c-\sigma) + c-\frac{1}{2}}\right]^{1/(c-\frac{1}{2})}.
\end{equation}
For sufficiently large $t$, let $C_{1}$ and $C_{2}$ be numbers satisfying
\begin{equation*}
|s-1|\geq C_{1}|s|; \qquad |s| \leq C_{2} |t|.
\end{equation*}
When $t>168\pi$ one can take $C_{1}^{-1} \geq 1 +\delta$ and $C_{2} \leq 1+\delta$, where $\delta = 2\cdot 10^{-6}$. This gives an estimate on the growth of $\zeta(s)$ in terms of $t$ only, and, together with (\ref{1.16}) proves
\begin{lemma}\label{l44}
Let $K,\theta$ and $t_{0}$ satisfy the relation that $|\zeta(\frac{1}{2}+it)| \leq Kt^{\theta}$ whenever $t>t_{0}>168\pi$. Also, let $\delta = 2\cdot 10^{-6}$ and let $c$ be a parameter satisfying $1<c\leq \frac{5}{4}$. Then throughout the region $\frac{1}{2}\leq \sigma\leq c$ the following estimate holds
\begin{equation*}
|\zeta(s)| \leq (1+\delta)\left\{K^{c-\sigma} \{\zeta(c)\}^{\sigma-\frac{1}{2}} ((1+\delta)\, t)^{\theta(c-\sigma)}\right\}^{1/(c-\frac{1}{2})}.
\end{equation*}
\end{lemma}
Now, in the integral
\begin{equation*}\label{convo}
\int_{\frac{1}{2} +it}^{\infty+it} \log |\zeta(s)| \, ds
\end{equation*}
one seeks to apply the convexity bound of Lemma \ref{l44} over the range $\frac{1}{2}\leq\sigma\leq c$, and to trivially estimate $\zeta(s)$ for $\sigma >c$. To this end, write
\begin{equation*}
\int_{\frac{1}{2} +it}^{\infty+it} \log |\zeta(s)| \, ds = \int_{\frac{1}{2}+ i t}^{c + it} \log |\zeta(s)|\, ds + m(c),
\end{equation*}
where
\begin{equation*}
m(c) := \int_{c+ i t}^{\infty + it} \log |\zeta(s)|\, ds \leq \int_{c}^{\infty} \log |\zeta(\sigma)|\, d\sigma,
\end{equation*}
since $c>1$. The application of Lemma \ref{l44} proves
\begin{lemma}\label{200}
Under the same assumptions as Lemma \ref{l44}, the following estimate holds,
\begin{equation*}
\Re\int_{\frac{1}{2}+it}^{\infty+i t} \log\zeta(s)\, ds < a_{1} + b_{1}\log t,
\end{equation*}
where
\begin{equation*}
a_{1} = \int_{c}^{\infty}\log|\zeta(\sigma)|\, d\sigma + \frac{1}{2}\left(c-\tfrac{1}{2}\right)\log \left\{K\zeta(c)\right\} + \delta,
\end{equation*}
and
\begin{equation*}
b_{1} =\frac{\theta}{2}(c-\tfrac{1}{2}).
\end{equation*}
\end{lemma}
The improvements in the following lemmas come from writing $\zeta(s+d)$ in place of $\zeta(s+1)$ which is used in the methods of Turing and Lehman. One then seeks the optimal value of $d\leq1$ at the end of the proof. Write
\begin{equation}\label{213}
\begin{split}
\Re\int_{\frac{1}{2}+it}^{\infty+it} \log \zeta(s)\, ds &= \int_{\frac{1}{2}+it}^{\frac{1}{2}+d+it}\log\bigg|\frac{\zeta(s)}{\zeta(s+d)}\bigg|\, ds\\ &+ \int_{\frac{1}{2}+d+it}^{\infty +it} \log |\zeta(s)|\, ds + \int_{\frac{1}{2}+d+it}^{\frac{1}{2}+2d+it} \log|\zeta(s)|\, ds,
\end{split}
\end{equation}
where $\frac{1}{2}<d\leq 1$. Since $d>\frac{1}{2}$ then $\Re (s) >1$ in the second and third integrals on the right side of the above equation. Thus, $\zeta(s)\geq \zeta(2\sigma)/\zeta(\sigma)$, so that, after suitable changes of variables, (\ref{213}) becomes
\begin{equation}\label{lattc}
\Re\int_{\frac{1}{2}+it}^{\infty+it}\log\zeta(s)\, ds \geq \int_{\frac{1}{2}+it}^{\frac{1}{2}+d+it}\log\bigg|\frac{\zeta(s)}{\zeta(s+d)}\bigg|\, ds + I(d),
\end{equation}
where
\begin{equation}\label{defineI}
\begin{split}
I(d) &= \frac{1}{2}\int_{1+2d}^{\infty}\log\zeta(\sigma)\, d\sigma - \int_{\frac{1}{2} +d}^{\infty}\log\zeta(\sigma)\, d\sigma\\
&+ \frac{1}{2}\int_{1+2d}^{1+4d}\log\zeta(\sigma)\, d\sigma - \int_{\frac{1}{2}+d}^{\frac{1}{2} +2d}\log\zeta(\sigma)\, d\sigma,
\end{split}
\end{equation}
and these integrals, all convergent, will be evaluated at the end of the proof.
The integrand the right side of (\ref{lattc}) can be rewritten\footnote{This method of approach is slightly easier than that given in Turing's paper, as noted by Lehman \cite[p.\ 310]{Lehman}.} using the Weierstrass product formula cf.\ \cite[pp.\ 82-83]{Davenport}
\begin{equation}\label{riewei}
\zeta(s) = \frac{e^{bs}}{2(s-1)\Gamma(1+\frac{s}{2})}\prod_{\rho}\left(1-\frac{s}{\rho}\right)e^{s/\rho},
\end{equation}
where the product is taken over zeroes $\rho$ and $b$ is a constant such that
\begin{equation*}\label{Davestyle}
b =\frac{1}{2}\log \pi -\sum_{\rho}\frac{1}{\rho},
\end{equation*}
when the sum converges if each zero is paired with its conjugate.
Thus
\begin{equation*}
\begin{split}
\log\bigg|\frac{\zeta(s)}{\zeta(s+d)}\bigg| &= \log\bigg|\frac{s+d-1}{s-1}\bigg|  - \log\bigg|\frac{\Gamma(\frac{s}{2} +1)}{\Gamma(\frac{s}{2} +1+\frac{d}{2})}\bigg| \\
&+\sum_{\rho}\log\bigg|\frac{s-\rho}{s+d-\rho}\bigg| - \frac{d}{2}\log \pi,
\end{split}
\end{equation*}
and so it follows that
\begin{equation}\label{L9final}
\begin{split}
\Re\int_{\frac{1}{2}+it}^{\infty+it} \log \zeta(s)\, ds &\geq  \sum_{\rho}\int_{\frac{1}{2}+it}^{\frac{1}{2}+d+it}\log\bigg |\frac{s-\rho}{s+d-\rho}\bigg |\, ds\\ 
&- \int_{\frac{1}{2}+it}^{\frac{1}{2}+d+it}\log \bigg | \frac{\Gamma(\frac{s}{2}+1)}{\Gamma(\frac{s+d}{2} +1)}\bigg |\, ds\\ &+ \int_{\frac{1}{2}+it}^{\frac{1}{2}+d+it} \log \bigg |\frac{s+d-1}{s-1}\bigg|\, ds +I(d) - \frac{d^{2}}{2}\log\pi \\& =I_{1} - I_{2} + I_{3} + I(d) - \frac{d^{2}}{2}\log \pi.
\end{split}
\end{equation}
The following Lemmas are needed for evaluation of $I_{1}$ and $I_{2}$. Since $\frac{1}{2}<d\leq 1$, it is easily seen that $I_{3}\geq 0$ but since the argument of the logarithm tends to one as $t \rightarrow\infty$, no further improvements are possible.
To estimate the integral $I_{2}$ the following result is required, which is a quantitative version of the classical estimate
\begin{equation*}
\frac{\Gamma'(z)}{\Gamma(z)} = \log z + O\left(\frac{1}{z}\right),
\end{equation*}
see, e.g.\ \cite[Ch.\ XII]{WW}.
\begin{lemma}\label{l8}
Define the symbol $\Theta$ in the following way: $f(x) = \Theta\{g(x)\}$ means that $|f(x)|\leq g(x)$. If $\Re z > 0$, then
\begin{equation*}
\frac{\Gamma '(z)}{\Gamma (z)} = \log z - \frac{1}{2 z} + \Theta\left(\frac{2}{\pi^{2} | (\Im z)^{2} - (\Re z)^{2}|}\right).
\end{equation*}
\end{lemma}
\begin{proof}
See \cite[Lem.\ 8]{Lehman}.
\end{proof}
Using the mean-value theorem for integrals, $I_{2}$ can be written as
\begin{equation*}
I_{2} = -\frac{1}{2}\int_{\frac{1}{2} +it}^{\frac{1}{2}+d+it}\left\{\int_{0}^{d} \Re \frac{\Gamma '\left(1+\frac{s+\xi}{2}\right)}{\Gamma\left(1+\frac{s+\xi}{2}\right)}\, d\xi\right\}\, ds = -\frac{d^{2}}{2} \Re \frac{\Gamma '\left(\sigma +\frac{it}{2}\right)}{\Gamma\left(\sigma +\frac{it}{2}\right)},
\end{equation*}
for some $\sigma: \frac{5}{4} <\sigma < d+ \frac{5}{4}$, whence by Lemma \ref{l8} 
\begin{equation}\label{finalI2}
I_{2} = -\frac{d^{2}}{2}\log \frac{t}{2} + \epsilon,
\end{equation}
where $|\epsilon|$ is comfortably less than $9\cdot 10^{-5}$ when $t>168\pi$ and decreases rapidly with increasing $t$.

In \cite{Turing}, the integrand in $I_{1}$ is evaluated using an approximate solution to a differential equation. This is then summed over the zeroes $\rho$. Using the fact that if $\rho = \beta + i\gamma$ lies off the critical line, then so too does $1-\overline{\rho}$, Booker \cite{Booker} was able to sharpen the bound on $I_{1}$. His result is given in the $d=1$ case of the following
\begin{lemma}[Booker]\label{l7}
Given a complex number $w$ with $|\Re(w)|\leq \frac{1}{2}$ then for $\frac{1}{2}< d\leq 1$ the following holds
\begin{equation*}
\int_{0}^{d} \log \bigg |\frac{(x+d+w)(x+d-\overline{w})}{(x+w)(x-\overline{w})}\bigg|\, dx \leq d^2 (\log 4) \Re\left(\frac{1}{d+w} + \frac{1}{d-\overline{w}}\right).
\end{equation*}
\end{lemma}
\begin{proof}
The proof for $d=1$ is given as Lemma 4.4 in \cite{Booker}. The adaptation to values of $d$ such that $\frac{1}{2}<d\leq 1$ is straightforward.
\end{proof}
Write
\begin{equation*}
I_{1} = \sum_{\rho}\int_{0}^{d}\log\bigg|\frac{\sigma +\frac{1}{2}+it-\rho}{\sigma+d + \frac{1}{2} + it-\rho}\bigg|\, d\sigma,
\end{equation*}
and apply Lemma \ref{l7} with $w=\frac{1}{2}+it-\rho$, pairing together $\rho$ and $1-\overline{\rho}$, whence
\begin{equation*}
I_{1} \geq -d^{2}(\log 4)\sum_{\rho}\Re\left(\frac{1}{\frac{1}{2}+d+it-\rho}\right).
\end{equation*}
Here the improvement of Booker's result is seen, as Lehman [\textit{op.\ cit.}]\ has $1.48$ in the place of $\log 4 \approx 1.38$. Rather than appealing to the Mittag-Leffler series for $\zeta(\frac{1}{2}+it)$ as in \cite{Lehman}, here one proceeds directly by rewriting the sum over the zeroes using the Weierstrass product (\ref{riewei}). By logarithmically differentiating (\ref{riewei}) and taking real parts, it is seen that
\begin{equation*}
\begin{split}
-I_{1} \leq d^{2}(\log 4)\bigg\{&\Re\frac{\zeta' \left(\frac{1}{2}+d+it\right)}{\zeta\left(\frac{1}{2}+d+it\right)} + \frac{d-\frac{1}{2}}{(d-\frac{1}{2})^{2} + t^{2}}\\
&+\frac{1}{2}\Re\frac{\Gamma'\left(\frac{1}{2}+\frac{\frac{1}{2} +d+it}{2}\right)}{\Gamma\left(\frac{1}{2}+\frac{\frac{1}{2} +d+it}{2}\right)} - \frac{1}{2}\log\pi\bigg\},
\end{split}
\end{equation*}
and thus, using Lemma \ref{l8} with $t>168\pi$ one has
\begin{equation*}
-I_{1} \leq d^{2}(\log 4)\left\{\Re\frac{\zeta'\left(\frac{1}{2}+d+it\right)}{\zeta\left(\frac{1}{2}+d+it\right)} + \frac{1}{2}\log\frac{t}{2} -\frac{1}{2}\log\pi + \epsilon'\right\},
\end{equation*}
where $|\epsilon'| \leq 10^{-4}$.
Finally, since $d>\frac{1}{2}$ then
\begin{equation*}
\Re\frac{\zeta'\left(\frac{1}{2}+d+it\right)}{\zeta\left(\frac{1}{2}+d+it\right)} \leq \bigg|\frac{\zeta'\left(\frac{1}{2}+d+it\right)}{\zeta\left(\frac{1}{2}+d+it\right)}\bigg| \leq -\frac{\zeta'\left(\frac{1}{2}+d\right)}{\zeta\left(\frac{1}{2}+d\right)},
\end{equation*}
and so
\begin{equation}\label{finalI1}
-I_{1} \leq d^{2} (\log 4)\left\{-\frac{\zeta'\left(\frac{1}{2}+d\right)}{\zeta\left(\frac{1}{2}+d\right)} + \frac{1}{2}\log  t -\frac{1}{2}\log{2\pi} +\epsilon'\right\}.
\end{equation}

The results for $I_{1}$ and $I_{2}$ contained in equations (\ref{finalI1}) and (\ref{finalI2}) respectively can be used in (\ref{L9final}) to give
\begin{equation*}
\begin{split}
-\Re\int_{\frac{1}{2}+it}^{\infty+it} \log \zeta(s)\, ds &\leq d^{2}(\log 4)\left\{-\frac{\zeta'\left(\frac{1}{2}+d\right)}{\zeta\left(\frac{1}{2}+d\right)} + \frac{1}{2}\log t -\frac{1}{2}\log 2\pi \right\}\\
& - \frac{d^{2}}{2}\log\frac{t}{2} - I(d) +\frac{d^{2}}{2}\log\pi +3\epsilon',
\end{split}
\end{equation*}
where $I(d)$ is defined by equation (\ref{defineI}). This then proves
\begin{lemma}\label{l9}
For $t>168\pi$, $d$ satisfying $\frac{1}{2}<d\leq 1$, and $\epsilon' =10^{-4}$, the following estimate holds
\begin{equation*}
-\Re\int_{\frac{1}{2}+it}^{\infty+it} \log \zeta(s)\, ds \leq a_{2} + b_{2}\log t,
\end{equation*}
where
\begin{equation*}
\begin{split}
a&= d^{2}(\log 4)\left\{-\frac{\zeta'(\frac{1}{2}+d)}{\zeta(\frac{1}{2}+d)} - \frac{1}{2}\log 2\pi +\frac{1}{4}\right\} + \frac{d^{2}}{2}\log\pi \\
&- \frac{1}{2}\int_{1+2d}^{\infty}\log\zeta(\sigma)\, d\sigma + \int_{\frac{1}{2} +d}^{\infty}\log\zeta(\sigma)\, d\sigma\\
&- \frac{1}{2}\int_{1+2d}^{1+4d}\log\zeta(\sigma)\, d\sigma + \int_{\frac{1}{2}+d}^{\frac{1}{2} +2d}\log\zeta(\sigma)\, d\sigma +3\epsilon',
\end{split}
\end{equation*}
and
\begin{equation*}
b= \frac{d^{2}}{2}\left(\log 4 -1\right).
\end{equation*}
\end{lemma}
Lemmas \ref{Ll1}, \ref{l44} and \ref{l9} prove at once
\begin{theorem}
Let $t_{2} >t_{1}>t_{0}>168\pi$ and let the pair of numbers $K,\theta$ satisfy the relation that $\zeta(\frac{1}{2}+it)\leq Kt^{\theta}$ for $t>t_{0}$. Also, let $\mu = 3 \cdot 10^{-6}$. If the parameters $c$ and $d$ are chosen such that $1<c\leq \frac{5}{4}$ and $\frac{1}{2}<d\leq 1$ then
\begin{equation*}
\bigg|\int_{t_{1}}^{t_{2}}S(t)\, dt\bigg| \leq a + b\log t_{2},
\end{equation*}
where 
\begin{equation}\label{Turingzeta}
\begin{split}
\pi a&= d^{2}(\log 4)\left\{-\frac{\zeta'(\frac{1}{2}+d)}{\zeta(\frac{1}{2}+d)} - \frac{1}{2}\log 2\pi + \frac{1}{4}\right\} +\frac{d^{2}}{2}\log\pi\\
&- \frac{1}{2}\int_{1+2d}^{\infty}\log\zeta(\sigma)\, d\sigma
+ \int_{\frac{1}{2} +d}^{\infty}\log\zeta(\sigma)\, d\sigma
- \frac{1}{2}\int_{1+2d}^{1+4d}\log\zeta(\sigma)\, d\sigma \\
&+\int_{\frac{1}{2}+d}^{\frac{1}{2} +2d}\log\zeta(\sigma)\, d\sigma
+ \frac{1}{2}(c-\tfrac{1}{2})\log \left\{K\zeta(c)\right\} + \int_{c}^{\infty}\log\zeta(\sigma)\, d\sigma + \mu,
\end{split}
\end{equation}
and
\begin{equation}\label{Turingzetab}
2\pi b= \theta(c-\tfrac{1}{2}) + d^{2}(\log 4 -1).
\end{equation}
\end{theorem}
\subsection{Calculations}\label{TuringCalc}
Taking the parameters $c=\frac{5}{4}$ and $d=1$, $\theta =\frac{1}{4}$ and $K=2.53$ one has, from the pair of equations (\ref{Turingzeta}) and (\ref{Turingzetab}) that $a=1.61$ and $b=0.0914$. These can be compared with the constants of Lehman, viz.\ $(a=1.7, b=0.114)$. It can also be seen that the minimal value of $b$ attainable by this method is $0.0353$.

Since the application of Turing's Method involves Gram blocks one wishes to minimise the bound given in (\ref{GB}). That is, one wishes to minimise the quantity $F(a, b, g_{p})$ given in (\ref{Ftimes}). Here, the values of $a$ and $b$ have been chosen to be optimal, for the application to Gram blocks, at height $g_{p}=2\pi\cdot 10^{12}$. Since it has been shown above that $a$ and $b$ are themselves functions of $c$ and $d$, write $F(c, d)$ for $F(a, b, 2\pi\cdot 10^{12})$.

Since there are no terms in (\ref{Turingzeta}) and (\ref{Turingzetab}) which involve both $c$ and $d$, one can write $F(c, d) = F_{c}(c) + F_{d}(d)$ and optimise each of the functions $F_{c}$ and $F_{d}$ separately. The presence of integrals involving the zeta-function in equations (\ref{Turingzeta}) and (\ref{Turingzetab}) makes the optimisation process difficult, even for a computer programme. Therefore, small values of $F_{c}(c)$ and $F_{d}(d)$ were sought over the intervals
\begin{equation*}\label{latticeTuring}
d= d(N)= 0.99- 2N\Delta; \quad c= c(N)= 1.24 - N\Delta,
\end{equation*}
where $\Delta = 0.02$ and $0\leq N\leq 12$. This showed that values of $F(c, d)\leq 3.72$ were clustered around $d=0.71$ and $c=1.08$. A further search for small values was conducted with
\begin{equation*}\label{latticeTuring2}
d= d(N)= 0.68+N\Delta; \quad c= c(N)= 1.05 + N\Delta,
\end{equation*}
where, this time, $\Delta = 0.01$ and $0\leq N\leq 20$. The smallest value found in this second search was $F(c, d) = 3.6805\ldots$, corresponding to $d= 0.74$ and $c=1.1$. For simplicity the choice of $d=\frac{3}{4}$ and $c=\frac{11}{10}$ gives $F(c, d)= 3.6812\ldots$ and obtains the constants in Theorem \ref{Thm1}, viz.\
\begin{equation*}
a(\tfrac{11}{10}, \tfrac{3}{4}) = 2.0666;\quad b(\tfrac{11}{10}, \tfrac{3}{4}) =0.0585.
\end{equation*}

\section{Dirichlet $L$-functions}\label{TMDLF}
\subsection{Introduction}\label{Outlineforothers}
In the works of Rumely \cite{Rumely} and Tollis \cite{Tollis}, analogues for Turing's Method are developed for Dirichlet $L$-functions, and for Dedekind zeta-functions respectively. Each of these proofs is based on \cite{Lehman}, so it is fitting to apply the above adaptations to yield better constants in these analogous cases. Since many of the details in the proofs are identical to those in \S\ref{TM}, this section and \S\ref{TMDZF} are less ponderous than the previous one.

\subsection{Analogies to the functions $Z(t)$, $\theta(t)$ and $S(t)$}
Let $\chi$ be a primitive Dirichlet character with conductor $Q>1$, and let $L(s, \chi)$ be the Dirichlet $L$-series attached to $\chi$. Furthermore define $\delta = (1-\chi(-1))/2$ so that $\delta$ is $0$ or $1$ according to whether $\chi$ is an even or odd character. Then the function
\begin{equation}\label{Dirichletdef}
\xi(s, \chi) = \left(\tfrac{Q}{\pi}\right)^{\frac{s}{2}} \Gamma\left(\tfrac{s+\delta}{2}\right) L(s, \chi)
\end{equation}
is entire and satisfies the functional equation
\begin{equation*}\label{Dirichletfunctional}
\xi(s, \chi)= W_{\chi}\xi(1-s, \overline{\chi}),
\end{equation*}
where
\begin{equation*}
W_{\chi} = i^{-\delta} \tau(\chi)Q^{-\frac{1}{2}}; \quad \tau(\chi) = \sum_{n=1}^{Q} \chi(n) e^{\frac{2\pi n i}{Q}}.
\end{equation*}
It is easily seen that $|W_\chi| = 1$ and so one may write $W_{\chi}= e^{i\theta_{\chi}}$ and, for $s=\frac{1}{2} +it$,
\begin{equation*}\label{Dirichlettheta}
\theta(t, \chi):= \frac{t}{2}\log\frac{Q}{\pi} + \Im\log\Gamma\left(\tfrac{s+\delta}{2}\right) - \frac{\theta_{\chi}}{2}.
\end{equation*}
Then the following functions $Z(t, \chi)$ and $\theta(t, \chi)$ are related by the equation
\begin{equation*}\label{DirichletRS}
Z(t, \chi) = e^{i\theta(t, \chi)}L(s, \chi),
\end{equation*}
where $Z(t, \chi)$ is real. This is analogous to the equation 
\begin{equation*}
Z(t) = e^{i\theta(t)}\zeta(\frac{1}{2} +it),
\end{equation*}
which can be found in \cite[Ch.\ IV, \S 17]{Titchmarsh}.

One can now show that $\theta(t, \chi)$ is ultimately monotonically increasing. This means that the \textit{Gram points} $g_{n}$ can be defined for Dirichlet $L$-functions as those points at which $\theta(g_{n}, \chi) = n\pi$.

Similarly to (\ref{sdef}), define, whenever $t$ is not an ordinate of a zero of $L(s, \chi)$, the function
\begin{equation}\label{DirichletS}
S(t, \chi) = \frac{1}{\pi}\arg L(\tfrac{1}{2}+it, \chi),
\end{equation}
where, as before, the argument is determined via continuous variation along the straight lines connecting $2$, $2+it$ and $\frac{1}{2} +it$; with a continuity condition if $t$ coincides with a zero.  It is known that 
\begin{equation*}\label{LittlewoodS}
\int_{t_{1}}^{t_{2}} S(t, \chi)\, dt = O(\log Qt_{2}),
\end{equation*}
and Turing's Method for Dirichlet $L$-functions requires a quantitative version of this result.
\subsection{Theorem and new results}\label{DirLResults}
\begin{theorem}\label{422}
Let $(a, b, t_{0})$ denote the following triple of numbers. Given $t_{0}>0$ there are positive constants $a$ and $b$ such that, whenever $t_{2}>t_{1}>t_{0}$ the following estimate holds
\begin{equation}\label{540}
\bigg| \int_{t_{1}}^{t_{2}} S(t, \chi)\, dt\bigg|\leq a+ b\log\frac{Qt_{2}}{2\pi},
\end{equation}
\end{theorem}
Rumely \cite{Rumely} has shown that $(1.8397, 0.1242, 50)$ satisfies (\ref{540}).
Analogous to Theorem \ref{LBrent} is 
\begin{theorem}[Rumely]\label{thmrum}
For $t_{2}>t_{1}>50$ the following estimate holds
\begin{equation}\label{Spidi}
\int_{t_{1}}^{t_{2}} \bigg|S(t, \chi) \frac{\theta'(t, \chi)}{\pi}\bigg|\leq 0.1592\log\frac{QT}{2\pi}\left(a + b\log\frac{QT}{2\pi}\right):= B(Q, t_{2}).
\end{equation}
\end{theorem}
The constant $0.1592$ comes from applying Stirling's formula to the function $\theta(t, \chi)$. It is this bound which is used in practical calculations. As in the case of the zeta-function, $a$ and $b$ are roughly inversely proportional, so one can choose these parameters in such a way that the quantity $B(Q, t_{2})$ is minimised for a given $Q$ and $t_{2}$.

At $Q=100$ and $t_{2} = 2500$, Rumely's constants $(a=1.8397, b=0.1242)$ give the value
\begin{equation*}
B(Q, t_{2}) \approx 5.32,
\end{equation*}
however there is a misprint in \cite{Rumely} and this is quoted as 4.824, which does not appear to affect his numerical calculations. The values of $a$ and $b$ have been optimised in (\ref{Spidi}) for $Q=100$ and $t_{2}=2500$, which proves the following

\begin{theorem}\label{TurDir}
If $t_{2}>t_{1}>t_{0}$ then the following estimate holds
\begin{equation*}
\bigg| \int_{t_{1}}^{t_{2}} S(t, \chi)\, dt\bigg|\leq 1.975+ 0.084\log\left(\frac{Qt_{2}}{2\pi}\right).
\end{equation*}
\end{theorem}
It therefore follows that $B(100, 2500) \approx 4.82$. Further reductions in the size of $B(Q, t_{2})$ are possible if the quantity $Qt_{2}$ is taken much larger, which will certainly happen in future calculations.

\subsection{Proof of Theorem \ref{TurDir}}\label{Proof of Theorem 1.3.2}
Littlewood's lemma on the number of zeroes of an analytic function in a rectangle is used to prove
\begin{lemma}\label{L563}
If $t_{2}>t_{1}>0$, then
\begin{equation}\label{563}
\int_{t_{1}}^{t_{2}} S(t, \chi)\, dt = \frac{1}{\pi}\int_{\frac{1}{2} +it_{2}}^{\infty+it_{2}} \log|L(s, \chi)|\, d\sigma - \frac{1}{\pi}\int_{\frac{1}{2} +it_{1}}^{\infty+it_{1}} \log|L(s, \chi)|\, d\sigma.
\end{equation}
\end{lemma}
\begin{proof}
The proof of this is the same as for Lemma \ref{Ll1}. 
\end{proof}
The following Lemma is a convexity estimate which will be used to give an upper bound on the first integral in (\ref{563}).
\begin{lemma}[Rademacher]\label{Rad567}
Suppose $1<c<\frac{3}{2}$. Then, for $1-c\leq \sigma \leq c$, for all moduli $Q>1$, and for all primitive characters $\chi$ with modulus $Q$,
\begin{equation}\label{Rademacher}
|L(s, \chi)| \leq \left(\frac{Q|1+s|}{2\pi}\right)^{\frac{c-\sigma}{2}} \, \zeta(c).
\end{equation}
\end{lemma}
\begin{proof}
See \cite[Thm 3]{Rademacher}.
\end{proof}
Rumely chooses $c= \frac{5}{4}$, but here the value of $c$ will be chosen optimally at the end of the argument. In preparation for taking the logarithm of both sides of (\ref{Rademacher}) note that for $\frac{1}{2}\leq \sigma \leq c$ and $t\geq t_{0}$, one can find an $\epsilon >0$ such that $\log(|1+s|/t) \leq \epsilon$. This will be used to express $|\log L(s, \chi)|$ as a function of $t$ rather than $s$. Indeed, if $\sigma \leq \frac{5}{4}$ and $t>t_{0}$ it is easy to show that
\begin{equation*}
\frac{|1+s|}{t} \leq 1 + \frac{81}{32t_{0}^{2}} = 1+\epsilon.
\end{equation*}

Write
\begin{equation}\label{5alive}
\int_{\frac{1}{2}+it}^{\infty +it}\log|L(s, \chi)|\, ds = \int_{\frac{1}{2}}^{c} \log|L(s, \chi)|\, d\sigma + \int_{c}^{\infty}\log|L(s, \chi)|\, d\sigma,
\end{equation}
where the convexity result will be applied to the first integral on the right-side. To estimate the second, note that for $\sigma \geq c>1$ one can write
\begin{equation*}
|L(s, \chi)| = \bigg|\sum_{n=1}^{\infty}\chi(n)n^{-s}\bigg| \leq \sum_{n=1}^{\infty} n^{-\sigma} = \zeta(\sigma).
\end{equation*}
With this estimation and the convexity estimate of (\ref{Rademacher}), equation (\ref{5alive}) becomes
\begin{equation*}
\begin{split}
\int_{\frac{1}{2}+it}^{\infty+it} \log|L(s, \chi)|\, d\sigma &\leq \frac{1}{4}\left(c-\tfrac{1}{2}\right)^{2} \left\{\log\frac{Qt}{2\pi} + \epsilon\right\} \\
&+ \left(c-\tfrac{1}{2}\right)\log\zeta(1+\eta) + \int_{c}^{\infty}\log|\zeta(\sigma)|\, d\sigma.
\end{split}
\end{equation*}
This then proves
\begin{lemma}\label{P1Di}
If $t>t_{0}>0$ and $c$ is a parameter satisfying $1<c\leq\frac{5}{4}$, then throughout the region $\frac{1}{2}\leq \sigma\leq c$ the following estimate holds
\begin{equation*}
 \int_{\frac{1}{2}+it}^{\infty+it} \log|L(s, \chi)|\, d\sigma \leq a_{1} + b_{1} \log\frac{Qt}{2\pi}.
\end{equation*}
where 
\begin{equation}\label{a609}
a_{1} = \frac{729}{2048t_{0}^{2}} + \left(c-\tfrac{1}{2}\right)\log\zeta(c) + \int_{c}^{\infty}\log|\zeta(\sigma)|\, d\sigma,
\end{equation}
and
\begin{equation*}
b_{1} =\frac{1}{4}(c-\tfrac{1}{2})^{2}.
\end{equation*}
\end{lemma}

Rumely uses $t_{0}=50$, whence one can take the first term in (\ref{a609}) to be at most $1.5\cdot 10^{-4}$. 

The improvements in the following Lemmas arise from taking $d$ to be in the range $\frac{1}{2}<d\leq 1$ and choosing the value of $d$ optimally at the end of the proof. One writes
\begin{equation*}
\int_{\frac{1}{2}+it}^{\infty+it} \log|L(s, \chi)|\, d\sigma
\end{equation*}
as a sum of integrals in the style of (\ref{213}). For $\sigma>1$ one can write
\begin{equation*}
\begin{split}
\log|L(s, \chi)| &= -\sum_{p}\log|1-\chi(p)p^{-s}| \geq-\sum_{p}\log(1+p^{-\sigma}) \\
&= \sum_{p}\left\{ -\log(1-p^{-2\sigma}) + \log(1-p^{-\sigma}) \right\}= \log \zeta(2\sigma) - \log\zeta(\sigma),
\end{split}
\end{equation*}
whence
\begin{equation*}
\int_{\frac{1}{2}+it}^{\infty+it} \log|L(s, \chi)|\, d\sigma \geq \int_{\frac{1}{2}+it}^{\frac{1}{2}+d+it} \log\bigg|\frac{L(s, \chi)}{L(s+d, \chi)}\bigg|\, d\sigma + I(d),
\end{equation*}
where $I(d)$ is the same function defined in (\ref{defineI}) in $\S \ref{Proof of Theorem 1.2.2}$.
Now the integrand on the right of the above equation can be expanded using the Weierstrass Product\footnote{Note that equation (18) of \cite{Rumely} has $(Q/\pi)^{s}$, rather than $(Q/\pi)^{s/2}$.}, see, e.g.\ \cite[pp.\ 84-85]{Davenport}
\begin{equation}\label{DWe}
\left(\tfrac{Q}{\pi}\right)^{\frac{s}{2}}\Gamma\left(\tfrac{s+\delta}{2}\right)L(s, \chi) = \xi(s, \chi) = e^{A+Bs}\prod_{\rho} (1-\tfrac{s}{\rho})e^{\frac{s}{\rho}},
\end{equation}
with
\begin{equation}\label{daveporto}
B = -\lim_{T\rightarrow\infty}\sum_{|\rho|<T}\frac{1}{\rho}.
\end{equation}
In the same manner as Turing's Method for the zeta-function, one arrives at
\begin{equation}\label{firstterm}
\begin{split}
\int_{\frac{1}{2}+it}^{\infty+it} \log|L(s, \chi)|\, d\sigma &\geq \frac{d^{2}}{2}\log\frac{Q}{\pi} + \int_{\frac{1}{2}+it}^{\frac{1}{2} +d+it} \log\bigg|\frac{\Gamma\left(\frac{s+d+\delta}{2}\right)}{\Gamma\left(\frac{s+\delta}{2}\right)}\bigg|\, d\sigma \\
&+\sum_{\rho}\int_{\frac{1}{2}+it}^{\frac{1}{2}+d+it}\log\bigg|\frac{s-\rho}{s+d-\rho}\bigg|\, d\sigma +I(d)\\
&= \frac{d^{2}}{2}\log\frac{Q}{\pi} +I_{1} +I_{2} +I(d).
\end{split}
\end{equation}
As before, one uses the second mean-value theorem for integrals to address $I_{1}$, whence 
\begin{equation*}
I_{1} = \int_{\frac{1}{2}+it}^{\frac{1}{2} +d+it} \log\bigg|\frac{\Gamma\left(\frac{s+d+\delta}{2}\right)}{\Gamma\left(\frac{s+\delta}{2}\right)}\bigg|\, d\sigma = \frac{d^{2}}{2}\Re\frac{\Gamma'\left(\tau+\frac{it}{2}\right)}{\Gamma\left(\tau+\frac{it}{2}\right)},
\end{equation*}
for
\begin{equation*}
\tau \in \left(\tfrac{1}{4}+\tfrac{\delta}{2},\, d +\tfrac{1}{4} + \tfrac{\delta}{2}\right) \subset \left(\tfrac{1}{4},\, d+\tfrac{3}{4}\right),
\end{equation*}
since $\delta$ is either $0$ or $1$. Using Lemma \ref{l8}, one has that
\begin{equation}\label{670}
I_{1} \geq \frac{d^{2}}{2}\left\{\log\frac{t}{2} - \epsilon'\right\},
\end{equation}
where 
\begin{equation}\label{670b}
\epsilon' = \frac{11}{t_{0}^{2}},
\end{equation}
and, since, $t_{0}>50$, it follows that $\epsilon'<5\cdot 10^{-3}$.
The application of Lemma \ref{l7} to $I_{2}$, with zeroes $\rho$ paired with $1-\overline{\rho}$ gives
\begin{equation*}\label{oncebooker}
I_{2} \geq -d^{2}(\log 4)\sum_{\rho}\Re\left(\frac{1}{d+\frac{1}{2} +it -\rho}\right).
\end{equation*}
Now logarithmically differentiate the Weierstrass product in (\ref{DWe}), take real parts, and use (\ref{daveporto}), to arrive at
\begin{equation}\label{almost}
\sum_{\rho}\Re\left(\frac{1}{s-\rho}\right) = \frac{1}{2}\log\frac{Q}{\pi} + \frac{1}{2} \Re\left(\frac{\Gamma'\left(\frac{s+\delta}{2}\right)}{\Gamma\left(\frac{s+\delta}{2}\right)}\right) + \Re\left(\frac{L'(s, \chi)}{L(s, \chi)}\right).
\end{equation}
For $\sigma=\Re(s) >1$, one can write
\begin{equation}\label{Rel}
\Re\left(\frac{L'(s, \chi)}{L(s, \chi)}\right) = \Re\left(\sum_{p}\frac{\chi(p)\log p}{p^{s}-\chi(p)}\right) \leq \sum_{p}\frac{\log p}{p^{\sigma} -1} = -\frac{\zeta'(\sigma)}{\zeta(\sigma)}.
\end{equation}
Thus when $s=d+\frac{1}{2}+it$, an application of Lemma \ref{l8} to (\ref{almost}) together with (\ref{Rel}) gives
\begin{equation}\label{suban}
I_{2} \geq -d^{2}(\log 4)\left(\frac{1}{2}\log\frac{Qt}{2\pi}  + \frac{5}{t_{0}^{2}} - \frac{\zeta'(\frac{1}{2} +d)}{\zeta(\frac{1}{2} +d)}\right).
\end{equation}
The results for $I_{2}$, contained in (\ref{suban}), and for $I_{1}$, contained in (\ref{670}) and (\ref{670b}), can be combined with (\ref{firstterm}) to prove
\begin{lemma}\label{wolfb}
For $t>t_{0}>50$ and for a parameter $d$ satisfying the condition $\frac{1}{2}<d\leq 1$, the following estimate holds
\begin{equation*}
-\int_{\frac{1}{2}+it}^{\infty+it}\log|L(s, \chi)|\, d\sigma \leq a_{2} + b_{2}\log\frac{Qt}{2\pi},
\end{equation*}
where
\begin{equation*}
\begin{split}
a &= \frac{13d^{2}}{t_{0}^{2}} - d^{2}(\log 4)\frac{\zeta'\left(\tfrac{1}{2}+d\right)}{\zeta\left(\tfrac{1}{2}+d\right)} - \frac{1}{2}\int_{2d+1}^{\infty}\log \zeta(\sigma)\, d\sigma\\ 
&+\int_{\frac{1}{2} +d}^{\infty}\log \zeta(\sigma)\, d\sigma-
+ \frac{1}{2}\int_{2d+1}^{4d+1}\log \zeta(\sigma)\, d\sigma + \int_{\frac{1}{2}+d}^{\frac{1}{2}+2d}\log \zeta(\sigma)\, d\sigma,
\end{split}
\end{equation*}
and
\begin{equation*}
b= \frac{d^{2}}{2}(\log 4 -1).
\end{equation*}
\end{lemma}
Lemmas \ref{L563}, \ref{P1Di} and \ref{wolfb} prove at once
\begin{theorem}
If $t_{2}>t_{1}>t_{0}>50$ and $c$ and $d$ are parameters such that $1<c\leq \frac{5}{4}$ and $\frac{1}{2}<d\leq 1$, the following estimate holds
\begin{equation*}
\bigg|\int_{t_{1}}^{t_{2}}S(t, \chi)\, dt\bigg| \leq a + b \log\left(\frac{Qt_{2}}{2\pi}\right),
\end{equation*}
where
\begin{equation}\label{finala}
\begin{split}
a\pi &= (c-\tfrac{1}{2})\log\zeta(c) + \int_{c}^{\infty}\log\zeta(\sigma)\, d\sigma - d^{2}(\log 4)\frac{\zeta'\left(\tfrac{1}{2}+d\right)}{\zeta\left(\tfrac{1}{2}+d\right)} \\
&- \frac{1}{2}\int_{2d+1}^{\infty}\log \zeta(\sigma)\, d\sigma 
+\int_{\frac{1}{2} +d}^{\infty}\log \zeta(\sigma)\, d\sigma\\
&- \frac{1}{2}\int_{2d+1}^{4d+1}\log \zeta(\sigma)\, d\sigma + \int_{\frac{1}{2}+d}^{\frac{1}{2}+2d}\log \zeta(\sigma)\, d\sigma +\frac{15d^{2}}{t_{0}^{2}}
\end{split}
\end{equation}
and 
\begin{equation}\label{finalb}
2b\pi = \frac{1}{2}\left(c-\tfrac{1}{2}\right)^{2} + d^{2}(\log 4 -1).
\end{equation}
\end{theorem}

\subsection{Calculations and improvements}\label{DirCalc}

In (\ref{finala}) and (\ref{finalb}) Rumely has $c=\frac{5}{4}$ and $d=1$ as well as $1.48$ in place of $\log 4 \approx 1.38$, and thus he calculates\footnote{The number $1.1242$ quoted by Rumely in his Theorem 2 is a result of a rounding error from his Lemma 2.}
\begin{equation*}
a = 1.839; \quad b= 0.1212.
\end{equation*}
Even with the same values of $c$ and $d$, the inclusion of Lemma \ref{l7} gives the result here that
\begin{equation*}
a(\tfrac{5}{4}, 1) = 1.794 \quad b(\tfrac{5}{4}, 1) = 0.1063.
\end{equation*}
For the values of $Q=100$, $t_{2}=2500$ the quantity $B(Q, t_{2})$ --- defined in Theorem \ref{thmrum} --- was minimised over two intervals using a computer programme, similarly to \S\ref{TuringCalc}. This yielded the optimal value for $B(Q, t_{2})$ at $c=1.17$ and $d=0.88$, whence the constants
\begin{equation*}
a(1.17, 0.88) = 1.9744, \quad b(1.17, 0.88) = 0.0833,
\end{equation*}
which appear in Theorem \ref{TurDir}.

\section{Dedekind zeta-functions}\label{TMDZF}
Let $K$ be a number field of degree $N$ with discriminant $D$ and with ring of integers $\mathcal{O}_{K}$. Let the signature of the field be $(r_{1}, r_{2})$, by which it is meant that $K$ has $r_{1}$ real embeddings and $r_{2}$ pairs of complex embeddings, whence $N=r_{1} +2r_{2}$. Then for $\Re(s) >1$ the Dedekind zeta-function is defined as
\begin{equation*}\label{dedekind}
\zeta_{K}(s) = \sum_{\mathfrak{a}\in \mathcal{O}_{K}} (N\mathfrak{a})^{-s} = \sum_{n\geq 1}a_{n} n^{-s},
\end{equation*}
where $\mathfrak{a}$ ranges over the non-zero ideals of $\mathcal{O}_{K}$ and $a_{n}$ is the number of ideals with norm $n$. Like the Riemann zeta-function, the Dedekind zeta-function can be extended via analytic continuation to the entire complex plane where it is defined as a meromorphic function with a simple pole at $s=1$. If 
\begin{equation*}\label{lamda}
\Lambda_{K}(s) = \Gamma(\tfrac{s}{2})^{r_{1}} \Gamma(s)^{r_{2}} \left(\frac{\sqrt{|D_{K}|}}{\pi^{\frac{N}{2}}2^{r_{2}}}\right)^{s} \zeta_{K}(s),
\end{equation*}
then the Dedekind zeta-function satisfies the functional equation
\begin{equation}\label{functional}
\Lambda_{K}(s) = \Lambda_{K} (1-s).
\end{equation}
One can define, see, e.g.\ \cite{Tollis}, the functions analogous to $Z(t)$ and $\theta(t)$ by
\begin{equation*}
Z_{K}(t) = e^{i\theta_{K}(t)}\zeta_{K}(\tfrac{1}{2}+it).
\end{equation*}
Analogous to the function $S(t)$ define,
\begin{equation*}\label{Sdef}
S_{K}(t) = \frac{1}{\pi} \arg \zeta_{K}(\tfrac{1}{2} +it); \quad S^{1}_{K}(t) = \int_{0}^{t}S_{K}(u)\, du,
\end{equation*}
where the valuation of the argument is determined, if $t$ is not an ordinate of a zero, by continuous variation along the line from $\infty +it$ to $\frac{1}{2} +it$ and $S(0) =0$. The modified Turing criterion for Dedekind zeta-functions relies on the following
\begin{theorem}\label{650}
Given $t_{0}>0$ there are positive constants $a, b$ and $g$ such that, whenever $t_{2}>t_{1}>t_{0}$ the following estimate holds
\begin{equation}\label{TuringTol}
\bigg|\int_{t_{1}}^{t_{2}} S_{K}(t)\, dt \bigg| \leq a+bN + g\log\left(|D_{K}| \left(\frac{t_{2}}{2\pi}\right)^{N}\right).
\end{equation}
\end{theorem}
If one denotes the quadruple $(a, b, g, t_{0})$ as those numbers satisfying (\ref{TuringTol}), then the work of Tollis \cite{Tollis} leads to the quadruple $(0.2627, 1.8392, 0.122, 40)$. 
Analogous to Theorem \ref{thmrum}, is the following
\begin{theorem}[Tollis]
For $t_{2}>t_{1}>40$ then
\begin{equation*}
\begin{split}
\bigg|\int_{t_{1}}^{t_{2}} S_{K}(t) \frac{\theta'_{K}(t)}{\pi}\, dt\bigg| &\leq \left(\frac{b}{2\pi}N + \frac{a}{2\pi}\right)\log\left(|D_{K}|\left(\frac{t_{2}}{2\pi}\right)^{N}\right)\\
&+ \frac{g}{2\pi}\log^{2}\left(|D_{K}|\left(\frac{t_{2}}{2\pi}\right)^{N}\right) \\
&= B(D_{K}, t_{2}, N).
\end{split}
\end{equation*}
\end{theorem}
For a given $D_{K}, t_{2}, N$ one wishes to choose the constants $a, b$ and $g$ so as to minimise $B(D_{K}, t_{2}, N)$.
For the sample values $N=4$, $D_{K}=1000$ and $t_{2} = 80$ one finds that Tollis's constants give $B(D_{K}, t_{2}, N)\approx 26.44$. As will be shown in \S \ref{Dedcalc}, very little improvement can be given on the constants of Tollis. Nevertheless the inclusion of Lemma \ref{l7} is enough to prove
\begin{theorem}\label{thm657}
Given $t_{2}>t_{1}>40$ the following estimate holds
\begin{equation}\label{Turing}
\bigg|\int_{t_{1}}^{t_{2}} S_{K}(t)\, dt \bigg| \leq 0.264+1.843N + 0.105\log\left(|D_{K}| \left(\frac{t_{2}}{2\pi}\right)^{N}\right).
\end{equation}
\end{theorem}
The improvements to Tollis's work will most likely be of use in the search for zeroes of Dedekind zeta-functions of large discriminant or degree but at \textit{small} height. For this reason the constant $t_{0}$ has been retained in the following equations, and appears in Theorem \ref{L890} from which Theorem \ref{thm657} is derived.
\subsection{Proof of Theorem \ref{thm657}}
As before, one begins by proving
\begin{lemma}\label{lwtollis}
\begin{equation*}
\pi\int_{t_{1}}^{t_{2}}S_{K}(t)\, dt = \int_{\frac{1}{2} +it_{2}}^{\infty +it_{2}} \log|\zeta_{K}(s)|\, ds - \int_{\frac{1}{2} +it_{1}}^{\infty +it_{1}} \log|\zeta_{K}(s)|\, ds.
\end{equation*}
\end{lemma}
\begin{proof}
The proof is the same as in Lemma \ref{Ll1}.
\end{proof}
The convexity estimate required is
\begin{lemma}[Rademacher]\label{6701}
For $1<c<\frac{3}{2}$ and $s= \sigma +it$ then throughout the range $1-c\leq\sigma\leq c$, the following estimate holds
\begin{equation}\label{convexity}
|\zeta_{K}(s)| \leq 3\bigg|\frac{1+s}{1-s}\bigg| \zeta(c)^{N}\left(|D_{K}|\left\{\frac{|1+s|}{2\pi}\right\}^{N}\right)^{\frac{c-\sigma}{2}}.
\end{equation}
\end{lemma}
\begin{proof}
See \cite[Thm 4]{Rademacher}.
\end{proof}
Note that, for $\frac{1}{2}\leq \sigma \leq c\leq \frac{5}{4}$ and for $t>t_{0}$ one can write
\begin{equation*}\label{thp}
\log|1+s|\leq \log t + \frac{81}{32t_{0}^{2}}.
\end{equation*}
This then enables one to place an upper bound on (\ref{convexity}) in term of $t$ rather than $s$.
Now write
\begin{equation*}
\int_{\frac{1}{2}+it}^{\infty +it}\log |\zeta_{K}(s)|\, ds = \int_{\frac{1}{2}+it}^{c +it}\log |\zeta_{K}(s)|\, ds + \int_{c}^{\infty} \log|\zeta_{K}(\sigma)|\, d\sigma,
\end{equation*}
where the second integral on the right-hand side is estimated trivially by the relation
\begin{equation}\label{pop}
\log|\zeta_{K}(\sigma+it)| \leq N\log\zeta(\sigma),
\end{equation}
since $\sigma >1$. The inequality in (\ref{pop}) can be seen by taking the prime ideal decomposition as in, e.g.\ \cite[p.\ 199]{Rademacher}. An application of the convexity estimates from Lemma \ref{6701} proves the following
\begin{lemma}\label{1stDed}
For $t>t_{0}>0$ and for a parameter $c$ satisfying $1<c\leq \frac{5}{4}$, the following estimate holds
\begin{equation*}
\int_{\frac{1}{2}+it}^{\infty+it} \log |\zeta_{K}(s)|\, ds \leq a_{1} + b_{1}N + g_{1}\log\left(|D_{K}|\left(\frac{t_{2}}{2\pi}\right)^{N}\right),
\end{equation*}
where
\begin{equation*}
a_{1}= \left(c-\tfrac{1}{2}\right)\left(\frac{81}{32t_{0}^{2}} + \log 3\right),
\end{equation*}
\begin{equation*}
b_{1}= \left(c-\tfrac{1}{2}\right)\left(\log\zeta(c) + \frac{81\left(c-\frac{1}{2}\right)}{128t_{0}^{2}}\right) + \int_{c}^{\infty}\log\zeta(\sigma)\, d\sigma,
\end{equation*}
and
\begin{equation*}
g_{1}=\frac{1}{4}\left(c-\tfrac{1}{2}\right)^{2}.
\end{equation*}
\end{lemma}
One writes
\begin{equation*}
\int_{\frac{1}{2}+it}^{\infty+it}\log|\zeta_{K}(s)|\, ds
\end{equation*}
as a sum of three integrals in the style of (\ref{213}). Thence, when $\sigma>1$ one can use the fact that
\begin{equation*}
\log|\zeta_{K}(s)| \geq N(\log|\zeta(2\sigma)| - \log|\zeta(s)|),
\end{equation*}
to write
\begin{equation*}\label{start}
\int_{\frac{1}{2}+it}^{\infty+it}\log|\zeta_{K}(s)|\, ds  \geq \int_{\frac{1}{2}+it}^{\frac{1}{2}+d+it}\log\bigg|\frac{\zeta_{K}(s)}{\zeta_{K}(s+d)}\bigg|\, ds +NI(d),
\end{equation*}
where $I(d)$ is the same function defined in (\ref{defineI}) in \S\ref{Proof of Theorem 1.2.2}. One aims at using the functional equation to estimate the integrand on the right-hand side. Using a result of Lang \cite[Ch.\ XIII]{Lang} one can write out the Weierstrass product viz.\
\begin{equation}\label{Lang}
s(s-1)\Lambda_{K} (s)= e^{a+bs}\prod_{\rho}\left(1-\frac{s}{\rho}\right)e^{s/\rho},
\end{equation}
where
\begin{equation*}
\Re b = -\sum_{\rho}\Re\frac{1}{\rho}.
\end{equation*} 
This then gives
\begin{equation}\label{Dedspli}
\begin{split}
\int_{\frac{1}{2}+it}^{\infty+it} \log|\zeta_{K}(s)|\, ds &\geq d^{2}\log\left(\frac{\sqrt{|D_{K}|}}{\pi^{\frac{N}{2}}2^{r_{2}}}\right) + r_{1}\int_{\frac{1}{2}+it}^{\frac{1}{2}+d+it}\log\bigg|\frac{\Gamma(\frac{s+d}{2})}{\Gamma(\frac{s}{2})}\bigg|\, ds \\
&+ r_{2}\int_{\frac{1}{2}+it}^{\frac{1}{2}+d+it}\log\bigg|\frac{\Gamma(s+d)}{\Gamma(s)}\bigg|\, ds \\
&+ \int_{\frac{1}{2}+it}^{\frac{1}{2}+d+it} \sum_{\rho}\log\bigg|\frac{s-\rho}{s+d-\rho}\bigg| \, ds + NI(d)\\
&\geq d^{2}\log\left(\frac{\sqrt{|D_{K}|}}{\pi^{\frac{N}{2}}2^{r_{2}}}\right) + I_{1}+I_{2}+I_{3}+NI(d).
\end{split}
\end{equation}

Applying the second mean-value theorem for integrals gives
\begin{equation*}
I_{1}= \frac{r_{1}d^{2}}{2} \Re\frac{\Gamma'\left(\frac{it}{2} +\tau_{1}\right)}{\Gamma\left(\frac{it}{2} +\tau_{1}\right)}; \quad I_{2} =  r_{2}d^{2} \Re\frac{\Gamma'\left(it +\tau_{2}\right)}{\Gamma\left(it +\tau_{2}\right)},
\end{equation*}
where $\frac{1}{4}<\tau_{1}<d+\frac{1}{4}$ and $\frac{1}{2} < \tau_{2} <2d+\frac{1}{2}$. Hence Lemma \ref{l8} gives
\begin{equation}\label{I1Ded}
I_{1} \geq \frac{r_{1}d^{2}}{2}\left(\log\frac{t}{2} - \frac{7}{2t_{0}^{2}}\right); \quad I_{2} \geq r_{2}d^{2}\left(\log t - \frac{11}{4t_{0}^{2}}\right).
\end{equation}
The integral $I_{3}$ is estimated using Lemma \ref{l7}; and logarithmic differentiation of (\ref{Lang}) gives
\begin{equation*}
\begin{split}
\sum_{\rho}\Re\left(\frac{1}{s-\rho}\right)&= \Re\left(\frac{1}{s} +\frac{1}{s-1}\right) + \frac{r_{1}}{2}\Re\frac{\Gamma'\left(\frac{s}{2}\right)}{\Gamma\left(\frac{s}{2}\right)}\\ 
&+ r_{2}\Re\frac{\Gamma'\left(s\right)}{\Gamma\left(s\right)} + \log\bigg|\frac{\sqrt{|D_{K}|}}{\pi^{\frac{N}{2}}2^{r_{2}}}\bigg| + \Re\left(\frac{\zeta'_{K}(s)}{\zeta_{K}(s)}\right).
\end{split}
\end{equation*}
Since
\begin{equation*}
\Re\frac{\zeta'_{K}(s)}{\zeta_{K}(s)} \leq -N\frac{\zeta'(\sigma)}{\zeta(\sigma)},
\end{equation*}
when $\sigma>1$, the expression for $I_{3}$ becomes
\begin{equation*}
\begin{split}
I_{3} \geq -d^{2}(\log 4)\bigg\{&\frac{2}{t_{0}^{2}} +\frac{r_{1}}{2}\left(\log\frac{t}{2} +\frac{2}{t_{0}^{2}}\right) + r_{2}\left(\log t + \frac{3}{2t_{0}^{2}}\right)\\
&+ \log\frac{\sqrt{|D_{K}|}}{\pi^{\frac{N}{2}}2^{r_{2}}} - N\frac{\zeta'(d+\tfrac{1}{2})}{\zeta(d+\tfrac{1}{2})}\bigg\}.
\end{split}
\end{equation*}
Thus the estimates for $I_{1}$ and $I_{2}$, which are contained in (\ref{I1Ded}) and the estimate of $I_{3}$ above prove, via (\ref{Dedspli})
\begin{lemma}\label{LastDed}
If $t>t_{0}>0$ and $d$ is a parameter that satisfies $\frac{1}{2}<d\leq 1$, then the following estimate holds
\begin{equation*}
-\int_{\frac{1}{2}+it}^{\infty +it} \log|\zeta_{K}(s)|\, ds \leq a_{2} + b_{2}N + g_{2}\log\left(|D_{K}|\left(\frac{t}{2\pi}\right)^{N}\right),
\end{equation*}
where
\begin{equation*}
a_{2}=\frac{4d^{2}\log 2}{t_{0}^{2}},
\end{equation*}
\begin{equation*}
b_{2}= d^{2}(\log 2)\left\{\log 2 - \frac{1}{2} - 2\frac{\zeta'\left(\tfrac{1}{2}+d\right)}{\zeta\left(\tfrac{1}{2}+d\right)} + \frac{8}{t_{0}^{2}}\right\}  - I(d),
\end{equation*}
and
\begin{equation*}
g_{2}= \frac{d^{2}}{2}(\log 4 -1),
\end{equation*}
and $I(d)$ is defined by (\ref{defineI}) in \S\ref{Proof of Theorem 1.2.2}.
\end{lemma}
Lemmas \ref{lwtollis}, \ref{1stDed} and \ref{LastDed} prove at once
\begin{theorem}\label{L890}
If $t_{2}>t_{1}>t_{0}>0$ and the parameters $c$ and $d$ satisfy $1<c\leq \frac{5}{4}$ and $\frac{1}{2}<d\leq 1$, then the following estimate holds
\begin{equation*}\label{turded}
\bigg|\int_{t_{1}}^{t_{2}}S_{K}(t)\, dt\bigg| \leq a + bN + g\log\left(|D_{K}|\left(\frac{t}{2\pi}\right)^{N}\right),
\end{equation*}
where
\begin{equation*}
\pi a= \left(c-\tfrac{1}{2}\right)\left(\frac{81}{32t_{0}^{2}} + \log 3\right) + \frac{4d^{2}\log 2}{t_{0}^{2}},
\end{equation*}
\begin{equation*}
\begin{split}
\pi b= &\left(c-\tfrac{1}{2}\right)\left(\log\zeta(c) + \frac{81\left(c-\frac{1}{2}\right)}{128t_{0}^{2}}\right) + \int_{c}^{\infty}\log\zeta(\sigma)\, d\sigma \\
&+ d^{2}(\log 2)\left\{\log 2 - \frac{1}{2} - 2\frac{\zeta'\left(\tfrac{1}{2}+d\right)}{\zeta\left(\tfrac{1}{2}+d\right)} + \frac{8}{t_{0}^{2}}\right\}  - I(d),
\end{split}
\end{equation*}
and
\begin{equation*}
\pi g= \frac{1}{4}\left(c-\tfrac{1}{2}\right)^{2} + \frac{d^{2}}{2}(\log 4 -1).
\end{equation*}
\end{theorem}
\subsection{Calculations}\label{Dedcalc}
Given the values $D_{K} = 1000, N=4, t_{0} = 40$ and $t_{2}=100$, the quantity to be minimised is
\begin{equation*}
F(a, b, g) = a + 4b + 18g,
\end{equation*}
with $a$, $b$ and $g$ defined in Theorem \ref{L890}. Proceeding with an optimisation programme similar to that in \S \ref{TuringCalc}, one finds that in fact the `trivial estimate', viz.\ the values $c=\frac{5}{4}$ and $d=1$ produce the minimum value of $F(a, b, g)$ and hence the minimum value of $B(D_{K}, t_{2}, N)$ as defined in (\ref{Turing}). The optimisation argument is only better than the trivial estimate when one of the parameters $D_{K}$, $t_{2}$ or $N$ is large, which will certainly occur in future calculations.

\section*{Acknowledgements}
My sincere thanks to Richard Brent, Herman te Riele and Sebastian Wedeniwski for their advice on computations using Turing’s method; and to Andrew Booker who recommended the writing of \S\S \ref{TMDLF} and \ref{TMDZF}. I am grateful for the kind suggestions of the referee. Lastly, I wish to thank my supervisor Roger Heath-Brown for his continual guidance and support.

\bibliography{themaster}
\bibliographystyle{plain}

\end{document}